\documentclass[reqno,a4paper,12pt]{amsart}
\usepackage[english]{babel}
\usepackage{url}
\usepackage
{hyperref}
\usepackage{comment}
\usepackage{tcolorbox}

\usepackage{cancel}

\newcommand{\crosses}[1]{%
	\ifcase#1\relax
	\or
	\rslash\or
	\rslash\mskip-5.5mu\rslash\or
	\rslash\mskip-5.5mu\rslash\mskip-5.5mu\rslash%
	\fi
}
\newcommand{\rslash}{\raisebox{.15ex}{/}}
\usepackage{xcolor,color}
\usepackage{geometry}
\usepackage[all, cmtip]{xy}
\usepackage{enumitem}
\usepackage{booktabs}
\usepackage{slashed}
\usepackage{bbm}
\usepackage{cancel}

\usepackage{cleveref}
\usepackage{tikz}
\usepackage{tikz-cd}
\tikzcdset{arrow style=math font, diagrams={>=stealth}}

\usepackage{xargs}  
\usepackage{soul}

\usepackage{amsmath}
\usepackage{amssymb,graphicx}
\usepackage{amsthm}
\usepackage{latexsym}
\usepackage{amsfonts}
\usepackage{bbm}
\usepackage{cases}

\calclayout

\numberwithin{equation}{section}

\theoremstyle{plain}
\newtheorem{lemma}{Lemma}[section]
\newtheorem{proposition}[lemma]{Proposition}
\newtheorem{proposition/definition}[lemma]{Proposition/Definition}
\newtheorem{theorem}[lemma]{Theorem}

\newtheorem*{theorem*}{Theorem}

\theoremstyle{definition}
\newtheorem{definition}[lemma]{Definition}
\newtheorem{propdef}[lemma]{Proposition / Definition}
\newtheorem{principle}[lemma]{Principle}
\newtheorem{remark}[lemma]{Remark}
\newtheorem{example}[lemma]{Example}

\DeclareRobustCommand{\SkipTocEntry}[5]{}

\title{The Symplectic-to-Contact Dictionary}

\makeatletter
\def\author@andify{%
	\nxandlist {\unskip ,\penalty-1 \space\ignorespaces}%
	{\unskip {} \@@and~}%
	{\unskip \penalty-2 \space \@@and~}%
}
\makeatother

\author[F.~Pugliese]{Fabrizio Pugliese}
\address{DipMat, Universit\`a degli Studi di Salerno, via Giovanni Paolo II n${}^{\circ}$123, 84084 Fisciano (SA), Italy.}
\email{\href{mailto:fpugliese@unisa.it}{fpugliese@unisa.it}}

\author[G.~Sparano]{Giovanni Sparano}
\address{DipMat, Universit\`a degli Studi di Salerno, via Giovanni Paolo II n${}^{\circ}$123, 84084 Fisciano (SA), Italy.}
\email{\href{mailto:sparano@unisa.it}{sparano@unisa.it}}

\author[L.~Vitagliano]{Luca Vitagliano}
\address{DipMat, Universit\`a degli Studi di Salerno, via Giovanni Paolo II n${}^{\circ}$123, 84084 Fisciano (SA), Italy.}
\email{\href{mailto:lvitagliano@unisa.it}{lvitagliano@unisa.it}}

\keywords{Contact geometry, symplectic geometry, almost contact geometry, $G$-structure}

\subjclass[2020]{53D10 (Primary), 
53D05, 
53D15, 
53D17, 
53C10, 
53C15
}

\begin{document}

\begin{abstract}
Contact Geometry is an odd dimensional analogue of Symplectic Geometry. This vague idea can actually be formalized in a rather precise way by means of a Symplectic-to-Contact Dictionary. The aim of this review paper is discussing the basic entries in this dictionary. Surprisingly, the dictionary can also be applied to apparently far away situations like complex and $G$-structures, to get old and new interesting geometries.
\end{abstract}

\maketitle

\tableofcontents

\section{Introduction}

Contact Geometry originated from the work of S. Lie on the Geometry of PDEs in the late 19th century (and later the work of Cartan and collaborators in the first half of the 20th century). Contact manifolds are the natural setting for a coordinate free theory of first order PDEs in 1 dependent variable (see, e.g., \cite[Section 3]{Vit2013}). Nowadays contact structures are mainly studied from two points of view:

\begin{itemize}
    \item Contact Metric Geometry: where contact (and related) structures are studied in connection with metric structures (fo instance Sasakian manifolds), see, e.g., \cite{Bla2010, Boy2008}, and
    \item Contact Topology: where global topological aspects of contact manifolds are analyzed (somewhat in parallel with Symplectic Topology), see, e.g., \cite{Gei2008}.
\end{itemize}

There are also works in Poisson Geometry studying contact structures in connection with various brackets, Lie algebroids, Lie groupoids \cite{Cra2007, Cra2013, Bla2020, Ang2024} and Geometric Mechanics (see \cite{Bra2017} for one of the first works in this direction, see also \cite{Gra2022} and references therein). This third line of research originated from the works of Kirillov \cite{Kir1976} and Lichnerowicz \cite{Lic1978}, and later Marle \cite{Mar1991} (and collaborators), on Jacobi brackets.

Contact manifolds are necessarily odd dimensional and Contact Geometry can be seen as an odd dimensional analogue of Symplectic Geometry in many respects. In fact, there is a motto by Arnol'd \cite{Arn1989}:

\begin{equation}\label{Arnol'd}
\begin{array}{c}
\text{\emph{Every theorem in symplectic geometry}} \\
\text{\emph{may be formulated as a contact geometry theorem}} \\
\text{\emph{and an assertion in contact geometry}} \\
\text{\emph{can be translated in the language of symplectic geometry.}}
\end{array}
\end{equation}

This is not just a motto. It turns out that there is a precise dictionary that allows to translate natural constructions and statements from Symplectic Geometry to Contact Geometry in a rather straightforward way. The aim of this note is reviewing this Symplectic-to-Contact Dictionary presenting its main entries in a concise way. Notice that parts of this Dictionary have already appeared in different papers in the literature, while a systematic presentation appears here for the first time. 

Somewhat surprisingly, the Dictionary can be applied to situations in Differential Geometry apparently far from Symplectic Geometry to get old and new interesting geometric structures. For instance, one can apply the dictionary to Complex Geometry \cite[Appendix A]{Sch2020} to find an odd dimensional analogue known in the literature as Normal Almost Contact Geometry \cite{Bla2010}. The dictionary can be also applied to Multisymplectic Geometry \cite{Vit2015}, Dirac Geometry \cite{Vit2018}, Generalized Complex Geometry \cite{Vit2016b}, and $G$-structures \cite{Tor2020}, among others.

The paper is organized as follows: in Section \ref{sec:contact} we recall the definition of a \emph{contact structure}. We never assume the existence of a global contact form, and this is an important feature that substantiates the Dictionary from the conceptual point of view. In Section \ref{sec:jacobi} we recall the Lie bracket on sections of the normal bundle $L$ to a contact structure. This bracket is the contact analogue of the Poisson bracket in Symplectic Geometry and it is equivalent to a tensor-like object analogous to the Poisson tensor which can easily be described in terms of the gauge algebroid and the first jet bundle of $L$. In Section \ref{sec:atiyah} we present a first version of the Symplectic-to-Contact Dictionary. The main characters here are the gauge algebroid $\mathsf D L$ of a line bundle $L$, and the so-called \emph{Atiyah forms} on $L$, i.e. $L$-valued, Lie algebroid cochains on $\mathsf D L$. Under the Dictionary, the tangent bundle translates to $\mathsf D L$, while differential forms translate to Atiyah forms. The Dictionary actually comes in two versions. In Section \ref{sec:homogeneous} we present the second version: contact structures are equivalent to \emph{homogeneous symplectic structures}. The first version of the Dictionary is more algebraic, while the second one is more geometric, but they are both useful. The last 3 sections provide examples of how to use the Dictionary to translate from specific geometries. In Section \ref{sec:jacobibundles} we apply the Dictionary to Poisson structures and find Jacobi structures. Section \ref{sec:jacobibundles}  also serves as an invitation to Jacobi Geometry. Finally, in Sections \ref{sec:oddcomplex} and \ref{sec:hgs} we apply the Dictionary to complex and $G$-structures, respectively. This latter instance reveals that there might actually be more than one dictionary, depending on the specific geometry we want to translate from. For instance, besides the Symplectic-to-Contact Dictionary, there is also a \emph{Symplectic-to-Cosymplectic Dictionary}, Cosymplectic Geometry being another interesting odd dimensional analogue of Symplectic Geometry.

We omit most of the proofs. The details can either be found in the cited references or easily recovered.

\bigskip

\noindent\textbf{Aknowledgments}. F. P. and L. V. are members of the GNSAGA of INdAM.

\section{\label{sec:contact} Contact Structures}

Let $M$ be a smooth manifold. A \textbf{contact structure} on $M$ is a hyperplane distribution $H \subseteq TM$ with the additional property of being maximally non-integrable. This means that the involutivity condition fails in the worst possible way. Concretely, given a
distribution $H \subseteq TM$, there is a well defined \textbf{curvature $2$-form}:
\begin{equation}
\kappa_H: \wedge^2 H \to TM/H, \quad
(X,Y) \mapsto [X,Y] \mathbin{\mathrm{mod}} H.
\end{equation}
The $2$-form $\kappa_H$ vanishes precisely when $H$ is involutive. When $H$ has codimension 1, then the normal bundle $L := TM/H$ is a line bundle and it makes sense to require $\kappa_H$ to be non-degenerate.

\begin{definition}
A \textbf{contact structure} on $M$ is a hyperplane distribution $H \subseteq TM$ such that $\kappa_H$ is non-degenerate.
\end{definition}

Let $H$ be a contact structure on $M$. From the non-degeneracy condition on $\kappa_H$ we immediately get that $\dim M = 2n + 1$ for some $n$, i.e.~$M$ is odd dimensional. Locally, $H$ is always the kernel of a $1$-form $\eta$, and the non-degeneracy condition on $H$ translates to the condition on $d\eta$ to be non-degenerate on $H$, or equivalently $\eta \wedge d\eta^n \ne 0$. One often assumes that the 1-form $\eta$ exists globally, and call $\eta$ a \textbf{contact 1-form}. This is equivalent to assume that the normal line bundle $L := TM/H$ is globally trivial. Every contact manifold has a 2-sheeted contact covering for which there exists a global contact 1-form. However, we prefer not to make any assumption on $L$ as we believe that this hides some important conceptual aspects. Additionally, there are some interesting contact manifolds for which there is no global contact $1$-form (see Example \ref{ex:proj}).

So let $(M,H)$ be a contact manifold (not necessarily equipped with a global $1$-form) and denote by $L := TM/H$ the normal bundle. We can still present $H$ as the kernel of a form-like object. Namely let
\begin{equation*}
\theta_H: TM \to L , \quad X \mapsto X \mathbin{\mathrm{mod}} H,
\end{equation*}
be the canonical projection. Clearly, $H = \ker \theta_H$. Additionally, $\theta_H$ can be understood as a 1-form with values in the line bundle $L$. Notice that, given any line bundle $L$ and any $L$-valued 1-form $\theta \in \Omega^1(M,L)$, the kernel $\ker \theta$ is a hyperplane distribution precisely when $\theta$ is everywhere non-zero. In Section \ref{sec:atiyah} we will discuss additional conditions on $\theta$ that guarantee that $\ker \theta$ is a contact distribution.

\begin{example}
On $\mathbb{R}^{2n+1}$ with coordinates $(x^i, u, p_i)$, $i = 1, \ldots, n$, consider the 1-form
\[
\eta_{\text{can}} = du - p_i dx^i.
\]
The kernel $H_{\text{can}}$ of $\eta_{\text{can}}$ is spanned by the vector fields
\[
\ldots, D_i := \frac{\partial}{\partial x^i} + p_i \frac{\partial}{\partial u}, \frac{\partial}{\partial p_i}, \ldots, \quad i = 1, \ldots, n,
\]
and it is a contact structure: the \textbf{standard contact structure}.
\end{example}

\begin{example}\label{ex:proj}
Let $N$ be a smooth manifold, with $\dim N = n > 1$, and let 
\[
M = \operatorname{Gr}(N, n-1)
\]
 be the Grassmannian bundle of hyperplanes in $N$. Equivalently, $M = P(T^\ast N)$, the projectivized cotangent bundle of $N$. Denote by $\pi: M \to N$ the projection. The manifold $M$ is canonically equipped with the line bundle $L \to M$ given by
\[
L_\xi := T_{\pi (\xi)} N/\xi, \quad \xi \in M,
\]
and there is a canonical surjection
\[
\theta: TM \to L, \quad 
v \mapsto d\pi(v) \mathbin{\mathrm{mod}} \xi,
\]
where $v \in T_\xi M$. The kernel $H$ of $\theta$ is a contact structure on $M$. When we interpret $M$ as the projectivized cotangent bundle, then $L$ is the dual $O(-1)$ of the tautological bundle on $P(T^\ast M)$. In particular, it is never trivial unless $n=2$.
\end{example}

Let $(M_1, H_1), (M_2, H_2)$ be contact manifolds. A diffeomorphism $\Phi: M_1 \to M_2$ identifying the contact structures is a \textbf{contactomorphism}. It is clear how to define symmetries and infinitesimal symmetries of a contact manifold $(M,H)$. An infinitesimal symmetry of $(M,H)$ is also called a \textbf{contact vector field}. The Lie algebra of contact vector fields of $(M,H)$ is denoted $\mathfrak{X}(M,H)$.

\begin{theorem}[Darboux Lemma]
Every contact manifold of dimension $2n+1$ is locally contactomorphic to $(\mathbb{R}^{2n+1}, H_{can})$.
\end{theorem}

As we mentioned in the introduction, contact manifolds can be understood as odd dimensional analogues of symplectic manifolds for several reasons. Some of these reasons will be discussed below. Actually, Sections \ref{sec:jacobi}--\ref{sec:homogeneous} aim at giving a rigorous conceptual meaning to Arnol'd's statement \eqref{Arnol'd} in the Introduction.

\section{\label{sec:jacobi} Jacobi Brackets and the Gauge Algebroid}
We begin showing that there is a Lie algebra naturally attached to a contact manifold, playing a role analogous to that of the Poisson algebra of (Hamiltonian) functions in Symplectic Geometry. We begin with a

\begin{lemma}
Let $(M,H)$ be a contact manifold and let $\theta_H: TM \to L = TM/H$ be the associated projection. The restriction
\begin{equation}\label{eq:theta_H}
\theta_H: \mathfrak{X}(M,H) \to \Gamma(L)
\end{equation}
is a vector space isomorphism (from contact vector fields to sections of the normal bundle $L$).
\end{lemma}

\begin{proof}
Recall that the curvature form $\kappa_H: \wedge^2 H \to L$ is non-degenerate. Hence it determines (mutually inverse) vector bundle isomorphisms
\[
\kappa_\flat: H \to H^\ast \otimes L, \quad \kappa^\sharp: H^\ast \otimes L \to H,
\]
by putting $\kappa_\flat(Y) := \kappa_H(Y,-)$. Now, let $X \in \mathfrak{X}(M,H)$ be such that $\theta_H(X) = 0$. This means that $X \in \Gamma(H)$ and we can compute $\kappa_\flat(X)$. For every $Z \in \Gamma(H)$ we have
\[
\kappa_\flat(X)(Z) = \kappa_H(X,Z) = \theta_H([X,Z]) = 0,
\]
where we used that $[X,Z] \in \Gamma(H)$ (as $X$ is a contact vector field and $Z$ is a section of $H$). So $\kappa_\flat(X) = 0$, hence $X = 0$. This shows that the linear map \eqref{eq:theta_H} is injective. For the surjectivity, let $\lambda \in \Gamma(L)$. Pick any vector field $X \in \mathfrak{X}(M)$ such that $\theta_H(X) = \lambda$ and consider the map
\[
\varphi_X: \Gamma(H) \to \Gamma(L), \quad Y \mapsto \theta_H([X,Y]).
\]
It is easy to see that $\varphi_X$ is $C^\infty(M)$-linear, hence it comes from a vector bundle map $\varphi_X: H \to L$, i.e. a section $\varphi_X$ of $H^\ast \otimes L$. Applying $\kappa^\sharp$, we get a vector field $Y = \kappa^\sharp (\varphi_X) \in \Gamma(H)$. Put $X_\lambda = X - Y$. Clearly $\theta_H(X_\lambda) = \lambda$. In order to conclude the proof it is enough to show that $X_\lambda \in \mathfrak{X}(M,H)$. In other words, we have to show that, for every $Z \in \Gamma(H)$, the commutator $[X_\lambda, Z]$ is in $H$ again, i.e. $\theta_H([X_\lambda, Z]) = 0$. So, compute
\begin{equation}
\theta_H([X_\lambda, Z]) = \theta_H([X, Z]) - \theta_H([Y, Z]) 
= \varphi_X(Z) - \kappa_H(Y,Z) = 0.
\end{equation}
This concludes the proof.
\end{proof}

The contact vector field $X_\lambda$ corresponding to the section $\lambda \in \Gamma(L)$ is sometimes called the \textbf{Hamiltonian vector field} with \textbf{Hamiltonian section} $\lambda$. Using the isomorphism $\mathfrak{X}(M,H) \cong \Gamma(L)$ we can transport the Lie bracket from $\mathfrak{X}(M,H)$ to $\Gamma(L)$. In this way we get a Lie bracket $\{-,-\}$ on $\Gamma(L)$. In other words
\[
\{\lambda, \mu\} = \theta_H([X_\lambda, X_\mu])
\]
or, equivalently,
\[
X_{\{\lambda, \mu\}} = [X_\lambda, X_\mu]
\]
for all $\lambda, \mu \in \Gamma(L)$. It is easy to see that the Lie bracket $\{-,-\}$ enjoys the following property
\[
\{\lambda, f\mu\} = X_\lambda(f)\mu + f\{\lambda, \mu\}
\]
for all $f \in C^\infty(M)$ and all $\lambda, \mu \in \Gamma(L)$. The bracket $\{-,-\}$ is called the \textbf{Jacobi bracket} of the contact manifold $(M,H)$.

\begin{definition}
A \textbf{derivation} of a vector bundle $E \to M$ is an $\mathbb{R}$-linear operator $D: \Gamma(E) \to \Gamma(E)$ enjoying the following Leibniz rule:
\[
D(fe) = X(f)e + fD(e),
\]
for some, necessarily unique, vector field $X \in \mathfrak{X}(M)$, all $f \in C^\infty(M)$, and all $e \in \Gamma(E)$. The vector field $X$ is also denoted $\sigma(D)$ and called the \textbf{symbol} of $D$.
\end{definition}

According to this definition, the Jacobi bracket $\{-,-\}$ is a derivation in one hence both arguments. For every section $\lambda \in \Gamma(L)$, the symbol of the derivation $\{\lambda, -\}$ is the Hamiltonian vector field $X_\lambda$.

Derivations of a vector bundle $E \to M$ are sections of a Lie algebroid $\mathsf D E \Rightarrow M$ sometimes called the \textbf{gauge algebroid} or \textbf{Atiyah algebroid} of $E$. The Lie bracket of derivations is just the commutator and the anchor map $\mathsf D E\to TM$ is the symbol.

\begin{remark}\label{rem:JT}
When $E = L$ is a line bundle, then any first order linear differential operator $\Gamma(L) \to \Gamma(L)$ is a derivation. Hence, in this case, $\mathsf D L$ is the dual module of the first jet bundle $\mathsf J^1 L$ twisted by $L$:
\[
\mathsf D L \cong \text{Hom}(\mathsf J^1 L, L)
\]
and, dually,
\[
\mathsf J^1 L \cong \text{Hom}(\mathsf D L, L).
\]
For instance, the Jacobi bracket $\{-,-\}: \Gamma(L) \times \Gamma(L) \to \Gamma(L)$ of a contact manifold can be seen as a skew-symmetric first order bidifferential operator or, equivalently, as an $L$-valued 2-form on the first jet bundle $\mathsf J^1 L$. We denote the latter $2$-form by
\[
J: \wedge^2 \mathsf J^1 L \to L,
\]
and sometimes call it the \textbf{Jacobi tensor} (because it plays a similar role as the Poisson tensor in Symplectic Geometry).
\end{remark}

\section{\label{sec:atiyah} Atiyah Forms and the Dictionary}
In this section we will put contact structures in a much more symplectic-like shape using the gauge algebroid described in the previous section. To do this we first discuss some basic properties of the gauge algebroid. Let $E \to M$ be a vector bundle and let $\mathsf D E\Rightarrow M$ be its gauge algebroid. We begin noticing that $\mathsf DE$ acts naturally on $E$: the action of a derivation on a section of $E$ is just the tautological one. The Lie algebroid structure on $\mathsf D E$ together with the action on $E$ defines a de Rham-like differential $d_{\mathsf D}$ on the graded vector space
\[
\Omega_{\mathsf D}^\bullet := \Gamma(\wedge^\bullet (\mathsf DE)^\ast \otimes E)
\]
of $E$-valued alternating forms on $\mathsf D E$. The differential of a $k$-form
\[
\omega: \wedge^k \mathsf D E \to E
\]
is defined by
\begin{align*}
d_{\mathsf D} \omega(D_1, \dots, D_{k+1}) &= \sum_{i=1}^{k+1} (-1)^{i+1} D_i (\omega(\dots, \hat{D}_i, \dots)) \\
&\quad + \sum_{i < j} (-1)^{i+j} \omega([D_i, D_j], \dots, \hat{D}_i, \dots, \hat{D}_j, \dots)
\end{align*}
for all $D_1, \dots, D_{k+1} \in \Gamma(\mathsf D E)$, where, as usual, a hat ``$\hat{-}$'' denotes omission. The cochain complex $(\Omega_{\mathsf D}^\bullet, d_{\mathsf D})$ is actually acyclic. To see this notice that the identity map $\mathbb I: \Gamma(E) \to \Gamma(E)$ is a derivation. A direct computation shows that the interior product
\[
\iota_{\mathbb I}: \Omega_{\mathsf D}^\bullet \to \Omega_{\mathsf D}^{\bullet-1}, \quad \omega \mapsto \iota_{\mathbb I} \omega = \omega( \mathbb I, -, \ldots, -)
\]
is a contracting homotopy for $(\Omega_{\mathsf D}^\bullet, d_{\mathsf D})$. 

Now, let $E = L$ be a line bundle. In this case, cochains in $(\Omega_{\mathsf D}^\bullet, d_{\mathsf D})$ will be called \textbf{Atiyah forms} (on $E$). It makes sense to define \emph{symplectic Atiyah forms}.

\begin{definition}
A \textbf{symplectic Atiyah form} (or \textbf{structure}) on $L$ is an Atiyah 2-form $\omega: \wedge^2 \mathsf D L \to L$ such that
\begin{enumerate}
    \item $\omega$ is $d_{\mathsf D}$-closed: $d_{\mathsf D} \omega = 0$, and
    \item $\omega$ is \textbf{non-degenerate}: the interior product $\omega_\flat: \Gamma(\mathsf D L) \to \Gamma(\text{Hom}(\mathsf D L,L)) = \Gamma(\mathsf J^1 L)$, $D \mapsto \iota_D \omega = \omega(D, -)$ is an isomorphism.
\end{enumerate}
\end{definition}

A line bundle equipped with a symplectic Atiyah form is a \textbf{symplectic Atiyah bundle}. Given symplectic Atiyah bundles $(L_1, \omega_1), (L_2, \omega_2)$, a \textbf{symplectic Atiyah morphism} $(L_1, \omega_1) \to (L_2, \omega_2)$ is a vector bundle isomorphism $L_1 \cong L_2$ identifying $\omega_1$ and $\omega_2$.

\begin{theorem}\label{CMvsSAB}
The category of contact manifolds and contactomorphisms is equivalent to the category of symplectic Atiyah bundles and symplectic Atiyah morphisms.
\end{theorem}

\begin{proof}[Proof (a sketch)]\quad

\textbf{From contact manifolds to symplectic Atiyah bundles.} Let $(M,H)$ be a contact manifold. Denote by $L = TM/H$ the normal bundle. We will construct a symplectic Atiyah form on $L$. The projection $\theta_H: TM \to L$ can be pre-composed with the symbol map $\sigma: \mathsf D L \to TM$ giving an Atiyah 1-form $\Theta = \theta_H \circ \sigma$. Taking the differential, we get a $d_{\mathsf D}$-closed Atiyah 2-form $\omega = d_{\mathsf D} \Theta$. The maximal non-integrability of $H$ now translates to $\omega$ being non-degenerate \cite{Vit2018}, so $\omega$ is a symplectic Atiyah form on $L$, as desired.

\textbf{From symplectic Atiyah bundles to contact manifolds.} Let $(L \to M, \omega)$ be a symplectic Atiyah bundle. As $d_{\mathsf D} \omega = 0$ and the interior product with the identity derivation $\mathbb{I}: \Gamma(L) \to \Gamma(L)$ is a contracting homotopy for the cochain complex $(\Omega_{\mathsf D}^\bullet, d_{\mathsf D})$, we get that $\Theta := \iota_\mathbb{I} \omega$ is a primitive for $\omega$, i.e. $\omega = d_{\mathsf D} \Theta$. From skew-symmetry, $\Theta(\mathbb{I}) = 0$. But (in this line bundle case) the derivation $\mathbb{I}$ generates the kernel of the symbol $\sigma: \mathsf D L \to TM$, hence $\Theta$ descends to an $L$-valued 1-form $\theta: TM \to L$ (the unique 1-form such that $\Theta = \theta \circ \sigma$). Put $H = \ker \theta$. The non-degeneracy of $\omega$ now translates to $H$ being a contact structure.
\end{proof}

As a first application of the language of symplectic Atiyah bundles we state without proof the following

\begin{proposition}\label{prop:J=w^-1}
Let $(M,H)$ be a contact manifold and let $(L \to M, \omega)$ be the associated symplectic Atiyah bundle. The Jacobi tensor
\[
J: \wedge^2 \mathsf J^1 L \to L
\]
of $(M,H)$ is the inverse of the symplectic Atiyah form
\[
\omega: \wedge^2 \mathsf D L \to L.
\]
\end{proposition}

Theorem \ref{CMvsSAB} says that contact structures are essentially the same as symplectic Atiyah forms. Moreover, together with Proposition \ref{prop:J=w^-1}, it suggests the following meta-mathematical principle that puts in a more conceptual framework Arnol'd statement:

\begin{principle}[\textbf{Symplectic-to-Contact Dictionary}] Every natural construction / statement in Symplectic (and symplectic-related) Geometry translates into an analogous construction / statement in Contact (and contact-related) Geometry. The translation is obtained via the following substitutions (that play the role of a \textbf{Symplectic-to-Contact Dictionary}):
\begin{equation*}
\begin{aligned}
C^\infty(M) & \leadsto \Gamma (L) \\
TM & \leadsto \mathsf DL \\
T^\ast M & \leadsto \mathsf J^1 L
\end{aligned}\quad .
\end{equation*}
\end{principle}

If we apply the Symplectic-to-Contact Dictionary to the definition of symplectic structure itself we obtain the definition of symplectic Atiyah bundle, showing how does the Dictionary work in the simplest possible case. In Sections \ref{sec:jacobibundles}--\ref{sec:hgs} we will present more applications of the Dictionary.

\section{\label{sec:homogeneous} Homogeneous Symplectic Structures}
There is another version of the Symplectic-to-Contact Dictionary that is often useful \cite{Bru2017,Tor2020,Vit2020,Gra2022} and we now discuss. We know already that line bundles with line bundle isomorphisms play an important role in the Dictionary. In some cases it is better to consider more general morphisms, namely line bundle maps that are only fiber-wise invertible. We have the following simple

\begin{proposition}\label{prop:LBvsHM}
The category of line bundles and fiber-wise invertible line bundle maps is equivalent to the category of principal $\mathbb{R}^\ast$-bundles and principal bundle maps (here $\mathbb{R}^\ast$ is the multiplicative Lie group of non-zero reals).
\end{proposition}

\begin{proof}
The equivalence is realized by the following obvious constructions. For every line bundle $L \to M$, we consider the principal $\mathbb{R}^\ast$-bundle $\widetilde{L} := L^\ast \smallsetminus 0$ (where the $\mathbb{R}^\ast$-action is given by fiber-wise scalar multiplication). In the other direction, for a principal $\mathbb{R}^\ast$-bundle $P \to M$ we consider the associated line bundle corresponding to the standard action of $\mathbb{R}^\ast$ on the 1-dimensional vector space $\mathbb{R}$. We leave it to the reader to check how these functors act on morphisms.
\end{proof}

In what follows, a principal $\mathbb{R}^\ast$-bundle $\widetilde{L}$ will be also called a \textbf{homogeneous manifold}. The $\mathbb{R}^\ast$-action will be also denoted $h$. This terminology / notation is motivated by the fact that $h$ allows to talk about \emph{homogeneous functions} on $\widetilde{L}$. Namely, a smooth function $f \in C^\infty(\widetilde{L})$ is \textbf{homogeneous of degree $p$} if
\[
h_t^\ast f = t^p f, \quad \text{for all } t \in \mathbb{R}^\ast.
\]

Proposition \ref{prop:LBvsHM} suggests that any natural construction involving line bundles can be rephrased using homogeneous manifolds. To make a first example, given a line bundle $L \to M$ and the associated homogeneous manifold $\widetilde{L}$, denote by $\pi: \widetilde{L} \to M$ the projection, and notice that a section $\lambda \in \Gamma(L)$ can be regarded as a smooth function on $\widetilde{L}$, denoted $\widetilde{\lambda}$, by putting
\[
\widetilde{\lambda}(\varepsilon) := \langle \varepsilon, \lambda_x \rangle,
\]
for all $\varepsilon \in \widetilde{L}$, where $x = \pi(\varepsilon)$ and $\langle -,-\rangle : L^\ast \times_M L \to \mathbb R$ is the duality pairing. The assignment $\lambda \mapsto \widetilde{\lambda}$ is a bijection between sections of $L$ and homogeneous functions of degree $1$. Similarly, for every derivation $D$ of $L$, there exists a unique vector field $\widetilde{D}$ on $\widetilde{L}$ such that
\[
\widetilde{D(\lambda)} = \widetilde{D} \widetilde{\lambda},
\]
for all $\lambda \in \Gamma(L)$. The vector field $\widetilde{D}$ is \emph{homogeneous of degree $0$} in the sense that
\[
h_t^\ast \widetilde{D} = \widetilde{D}, \quad \text{for all } t \in \mathbb{R}^\ast,
\]
and the assignment $D \mapsto \widetilde{D}$ is a bijection between derivations of $L$ and homogeneous vector fields of degree $0$ on $\widetilde{L}$. Finally, for every Atiyah $k$-form $\omega \in \Omega_{\mathsf D}^\bullet$, there exists a unique differential $k$-form $\widetilde{\omega}$ on $\widetilde{L}$ such that
\[
\widetilde{\omega}(\widetilde D_1, \dots, \widetilde D_k) = \widetilde{\omega(D_1, \dots, D_k)},
\]
for all $D_1, \dots, D_k \in \Gamma(\mathsf D L)$. The $k$-form $\widetilde{\omega}$ is \emph{homogeneous of degree $1$}, i.e.
\[
h^\ast_t \widetilde{\omega} = t \widetilde{\omega}, \quad \text{for all } t \in \mathbb{R}^\ast,
\]
and the assignment $\omega \mapsto \widetilde{\omega}$ is a bijection between Atiyah forms on $L$ and homogeneous differential forms of degree $1$ on $\widetilde{L}$. We are now ready to state the main result of this section. We begin with a

\begin{definition}
A \textbf{homogeneous symplectic structure} on a homogeneous manifold $\widetilde{L}$ is a symplectic form $\widetilde{\omega}$ on $\widetilde{L}$ which is additionally homogeneous of degree $1$. A homogeneous manifold equipped with a homogeneous symplectic structure is a \textbf{homogeneous symplectic manifold}.
\end{definition}

\begin{theorem}\label{theor:CMvsHSM}
The category of contact manifolds and contactomorphisms is equivalent to the category of homogeneous symplectic manifolds and $\mathbb R^\ast$-equivariant symplectomorphisms.
\end{theorem}

\begin{proof}[Proof (a sketch, see also \cite{Bru2017})]\quad

\textbf{From contact manifolds to homogeneous symplectic manifolds.} Let $(M,H)$ be a contact manifold and let $(L,\omega)$ be the corresponding symplectic Atiyah bundle. The symplectic Atiyah structure $\omega$ determines a homogeneous 2-form $\widetilde{\omega}$ of degree $1$ on $\widetilde{L}$. It is easy to check that $d\widetilde{\omega} = \widetilde{d_{\mathsf D} \omega} = 0$, i.e. $\widetilde{\omega}$ is a closed 2-form. Finally, as $\omega$ is non-degenerate, $\widetilde{\omega}$ is non-degenerate as well. In other words it is a homogeneous symplectic structure on $\widetilde{L}$.

\textbf{From homogeneous symplectic manifolds to contact manifolds.} Let $(\widetilde{L}, \widetilde{\omega})$ be a homogeneous symplectic manifold. The homogeneous symplectic structure $\widetilde{\omega}$ comes from an Atiyah 2-form $\omega$ on $L$. As $\widetilde{\omega}$ is symplectic, $\omega$ is symplectic as well, i.e. $(L,\omega)$ is a symplectic Atiyah bundle. Finally, $(L,\omega)$ corresponds to a contact manifold $(M,H)$.
\end{proof}

The homogeneous symplectic manifold $(\widetilde{L}, \widetilde{\omega})$ associated to a contact manifold $(M,H)$ is sometimes called the \textbf{symplectization} (or the \textbf{symplectic cone}) of $(M,H)$.

\begin{proposition}\label{prop:JBvsHPB}
Let $(M,H)$ be a contact manifold and let $(\widetilde{L}, \widetilde{\omega})$ be its symplectization. The Jacobi bracket
\[
\{-,-\}: \Gamma(L) \times \Gamma(L) \to \Gamma(L)
\]
of $(M,H)$ and the Poisson bracket $\{-,-\}_{\widetilde{\omega}}$ of $(\widetilde{L}, \widetilde{\omega})$ are related by the following formula
\[
\{\widetilde{\lambda}, \tilde{\mu}\}_{\widetilde \omega} = \widetilde{\{\lambda, \mu\}}, \quad \text{for all } \lambda, \mu \in \Gamma(L).
\]
Equivalently, the Jacobi tensor
\[
J: \wedge^2 \mathsf J^1 L \to L
\]
and the Poisson tensor
\[
\widetilde J = \widetilde{\omega}{}^{-1} \in \Gamma(\wedge^2 T\widetilde{L})
\]
are related by
\[
\langle \widetilde{J}, \widetilde{\psi} \wedge \widetilde{\chi} \rangle = \widetilde{\langle J, \psi \wedge \chi \rangle}, \quad \text{for all } \psi, \chi \in \Gamma(\mathsf J^1 L) = \Omega_{\mathsf D}^1.
\]
\end{proposition}

Theorem \ref{theor:CMvsHSM} says that contact structures are essentially the same as homogeneous symplectic structures. Moreover, together with Proposition \ref{prop:JBvsHPB}, it suggests the following new

\begin{principle}[\textbf{Symplectic-to-Contact Dictionary 2}] Every natural construction / statement in Symplectic (and symplectic-related) Geometry translates into an analogous construction / statement in Contact (and contact-related) Geometry. The translation is obtained replacing ordinary manifolds with homogeneous manifolds, and implementing the correct homogeneity properties.
\end{principle}

If we apply the Symplectic-to-Contact Dictionary 2 to the definition of symplectic manifold itself we obtain the definition of homogeneous symplectic manifold. The advantage of Dictionary 2 over Dictionary 1 is that it makes you deal with ordinary structures (vector fields, differential forms, etc.) instead of their Atiyah analogues (derivations, Atiyah forms, etc.). The disadvantage of Dictionary 2 is that one has to guess what the correct homogeneity property is in each case separately (although the obvious relation between the two dictionaries allows to make the correct guess in most cases).

\section{\label{sec:jacobibundles} Jacobi bundles}
We begin this section applying the Symplectic-to-Contact Dictionary to Poisson structures. We first
adopt version 1 of the Dictionary. If we do so, we find \emph{Jacobi structures}.

\begin{definition}\label{def:Jacobi}
A \textbf{Jacobi bundle} is a line bundle $L \to M$ equipped with a \textbf{Jacobi
bracket}, i.e. a Lie bracket
\[
\{-,-\}: \Gamma(L) \times \Gamma(L) \to \Gamma(L)
\]
which is additionally a derivation in each argument. A \textbf{Jacobi manifold} is a manifold
equipped with a Jacobi bundle. A \textbf{Jacobi map} between the Jacobi bundles $(L_1, \{-,-\}_1)$, $(L_2, \{-,-\}_2)$ is a fiber-wise invertible line bundle map $F : L_1 \to L_2$ such that $\{F^\ast \lambda, F^\ast \lambda'\}_1 = F^\ast \{ \lambda, \lambda'\}_2$, for all $\lambda,\lambda' \in \Gamma (L_2)$.
\end{definition}

Adopting version 2 of the Dictionary we get essentially the same answer. More precisely, 
if we apply Dictionary 2 to Poisson manifolds, we get \emph{homogeneous Poisson
manifolds}.

\begin{definition}\label{def:HPM}
A \textbf{homogeneous Poisson manifold} is a homogeneous manifold $\widetilde{L}$
equipped with a \textbf{homogeneous Poisson bracket}, i.e. a Poisson bracket
\[
\{-,-\}_{\widetilde{L}}: C^\infty(\widetilde{L}) \times C^\infty(\widetilde{L}) \to C^\infty(\widetilde{L})
\]
which is additionally homogeneous of degree $-1$:
\[
h^\ast_t \{-, -\}_{\widetilde{L}} = t^{-1}\{-, -\}_{\widetilde{L}}, \quad \text{for all } t \in \mathbb{R}^\ast,
\]
or, equivalently, $\{-, -\}_{\widetilde{L}}$ preserves homogeneous functions of degree $1$.
\end{definition}

Notice that the homogeneity condition on $\{-, -\}_{\widetilde{L}}$ in Definition \ref{def:HPM} is actually an
educated guess. For instance it is motivated by the following

\begin{theorem}
The category of Jacobi manifolds and Jacobi maps is equivalent to the
category of homogeneous Poisson manifolds and $\mathbb R^\ast$-equivariant Poisson maps.
\end{theorem}

\begin{proof}[Proof (a sketch)]
We will not discuss morphisms. Let $L \to M$ be a line bundle. For
every Jacobi structure $\{-, -\}$ on $L$ there exists a unique Poisson structure $\{-, -\}_{\widetilde L}$ on
$\widetilde{L}$ such that
\[
\{\widetilde{\lambda}, \widetilde{\mu}\}_{\widetilde L} = \widetilde{\{\lambda, \mu\}}, \quad \text{for all } \lambda, \mu \in \Gamma(L).
\]
The Poisson bracket $\{-, -\}_{\widetilde L}$ is homogeneous of degree $-1$, and the
assignment $\{-, -\} \mapsto \{-, -\}_{\widetilde L}$ establishes a bijection between Jacobi structures on $L$
and homogeneous Poisson brackets of degree $-1$ on $\widetilde{L}$.
\end{proof}

The homogeneous Poisson manifold corresponding to a Jacobi manifold is sometimes called the \textbf{Poissonization}.

The original definition of Jacobi manifolds goes back to the works of Kirillov and,
independently, Lichnerowicz in the late 70s \cite{Kir1976,Lic1978}. Despite Jacobi structures and Poisson
structures were introduced essentially at the same time (in their contemporary form),
the mathematical community concentrated much more on Poisson Geometry than on
Jacobi Geometry. In order to illustrate the wide spectrum of applications of
Jacobi Geometry, we now present a list of examples. Contact / Jacobi Geometry has also
applications in time dependent mechanics and mechanics of dissipative systems \cite{Bra2017,Gra2022} but we will
not discuss this.

\begin{example} \label{ex:2.4}
Let $(M, L, \{-, -\})$ be a Jacobi manifold. If $L = \mathbb R_M := M \times \mathbb R \to M$
is the trivial line bundle, then $\Gamma(L) \cong C^\infty(M)$ and the Jacobi bracket $\{-,-\}$ is
necessarily of the form
\begin{equation}\label{eq:JB_tlb}
\{f,g\} = \Lambda(df, dg) + E(f)g - E(g)f
\end{equation}
for some bivector $\Lambda \in \Gamma(\wedge^2 TM)$ and some vector field $E \in \mathfrak{X}(M)$ such that
\begin{equation}\label{eq:JP}
[\Lambda, \Lambda] = 2 E \wedge \Lambda, \quad \text{and} \quad [E, \Lambda] = 0,
\end{equation}
where $[-,-]$ is the Schouten-Nijenhuis bracket of multivectors. Conversely, given a
bivector $\Lambda$ and a vector field $E$ satisfying \eqref{eq:JP}, the bracket given by \eqref{eq:JB_tlb} is a Jacobi
bracket on $\mathbb R_M$.
\end{example}

\begin{example} \label{ex:2.5}
Let $(M, H)$ be a contact manifold, let $(L, \omega)$ be the associated symplectic Atiyah
bundle, and let $\{-, -\}$ be the Jacobi bracket induced by $H$. Then $\{-, -\}$ is
a Jacobi bracket in the sense of Definition \ref{def:Jacobi}. Conversely, let $(L, \{-, -\})$ be a Jacobi
bundle. Exactly as in Section \ref{rem:JT}, the bracket $\{-,-\}$ is equivalent to a Jacobi
tensor:
\[
J : \wedge^2 J^1 L \to L.
\]
It is easy to see that $(L, \{-, -\})$ is the Jacobi bundle of a contact manifold $(M, H)$ if and only if 
$J$ is nondegenerate, in which case $J = \omega^{-1}$ where $\omega$ is the symplectic Atiyah structure
determined by $H$. In particular, the Poissonization of $(M, L, \{-, -\})$ agrees with the Poisson structure associated to the
symplectization of $(M, H)$.
\end{example}

\begin{example} \label{ex:2.6}
Let $(M, \omega, \eta)$ be a \emph{locally conformally symplectic} (\emph{lcs}) \emph{manifold}. Recall
that this means that $\omega$ is a non-degenerate $2$-form on $M$, and $\eta$ is a closed $1$-form such
that $d\omega = \eta \wedge \omega$. The bracket
\begin{equation}\label{eq:lcsB1}
\{f, g\} = \omega^{-1}(df - f\eta, dg - g\eta), \quad f, g \in C^\infty(M),
\end{equation}
is a Jacobi bracket on $\mathbb R_M$. There is a slightly more intrinsic version of this example \cite[Appendix A]{Vit2016}.
Namely let $L \to M$ be a line bundle. A \textbf{lcs structure} on $L$ is a pair $(\omega, \nabla)$ consisting
of a non-degenerate $L$-valued 2-form $\omega \in \Omega^2(M, L)$ and a flat connection $\nabla$ in $L$, such
that $d^\nabla \omega = 0$, where $d^\nabla$ is the connection differential. The bracket
\begin{equation}\label{eq:lcsB2}
\{\lambda, \mu\} = \omega^{-1}(d^\nabla \lambda, d^\nabla \mu), \quad \lambda, \mu \in \Gamma(L),
\end{equation}
is a Jacobi bracket on $L$. When $L = \mathbb R_M$, then $\omega$ is an ordinary 2-form and $\nabla$ is
completely determined by its connection 1-form $\eta$. Actually lcs structures on $\mathbb R_M$ are
the same as ordinary lcs structures, and in this case, Formulas \eqref{eq:lcsB1} and \eqref{eq:lcsB2} agree.
\end{example}

\begin{example} \label{ex:2.7}
Let $\mathfrak{g}$ be a (real, finite dimensional) Lie algebra. The dual $\mathfrak{g}^\ast$ is equipped
with the usual Lie-Poisson structure $\{-,-\}_{\mathfrak{g}^\ast}$. The slit dual $\widetilde L := \mathfrak{g}^\ast \smallsetminus 0$ is actually
a homogeneous Poisson manifold. Namely, the usual scalar multiplication by non-zero reals gives to $\widetilde L$ the structure of a principal $\mathbb{R}^\ast$-bundle over the projectivized
dual $P(\mathfrak{g}^\ast)$, and the restricted Lie-Poisson structure is homogeneous of degree $-1$.
Notice that the line bundle corresponding to the homogeneous manifold $\widetilde L$ is the dual
$\mathcal{O}(-1) \to P(\mathfrak{g}^\ast)$ of the tautological bundle on the projective space $P(\mathfrak{g}^\ast)$. We conclude that the dual of the tautological bundle on the projectivized dual of any (real, finite
dimensional) Lie algebra is a Jacobi bundle.
\end{example}

\begin{example} \label{ex:2.8}
Let $\mathfrak{g}$ be a (real, finite dimensional) Lie algebra, and let $\varphi \in \mathfrak{g}^\ast$ be a
$1$-cocycle, i.e. $\varphi$ vanishes on the commutator subalgebra $[\mathfrak{g}, \mathfrak{g}] \subseteq \mathfrak{g}$. Then $\mathfrak{g}^\ast$ can be given the structure of a Jacobi manifold with $L = \mathbb R_{\mathfrak{g}^\ast}$, the trivial line bundle, and the
Jacobi bracket
\[
\{-,-\} : C^\infty(\mathfrak{g}^\ast) \times C^\infty(\mathfrak{g}^\ast) \to C^\infty(\mathfrak{g}^\ast)
\]
uniquely determined by the following conditions:
\begin{align*}
\{\ell_v, \ell_w\} &= \ell_{[v,w]}, \\
\{\ell_v, c\} &= \varphi(v)c, \\
\{c, c'\} &= 0,
\end{align*}
for all $v, w \in \mathfrak{g}$, and $c, c'$ constant, where, for $v \in \mathfrak{g}$, we denoted by $\ell_v$ the corresponding
linear function on $\mathfrak{g}^\ast$.
\end{example}

The latter two examples generalize to Lie algebroids.

\begin{example} \label{ex:2.9}
$b$-contact structures are mildly singular contact structures and were
recently considered by Miranda and Oms \cite{Mir2023} (see also \cite{Car2022}). Specifically, a $b$-contact manifold is a $(2n+1)$-dimensional Jacobi manifold $(M, L, \{-, - \})$ such that the Jacobi tensor $J$ is everywhere
non-degenerate, except for a hypersurface $Z \subseteq M$ where it satisfies a certain transversality condition. In particular, $M \smallsetminus Z$ is a contact manifold (but the contact structure has a ``singularity'' along $Z$). 
\end{example}

\section{\label{sec:oddcomplex}Odd Complex Structures}

Interestingly, the Symplectic-to-Contact Dictionary can also be applied to situations apparently very far from Symplectic Geometry. In this section we apply it to
complex manifolds. We begin adopting Dictionary 1. In this case we get the following

\begin{definition} \label{def:2.10}
A \textbf{complex Atiyah bundle} is a (real) line bundle $L \to M$ equipped
with a \textbf{complex Atiyah structure}, i.e. a fiber-wise complex structure
\[
K : \mathsf D L \to \mathsf D L
\]
on the gauge algebroid $\mathsf D L \Rightarrow M$ satisfying an additional integrability condition, namely
the \textbf{Lie algebroid Nijenhuis torsion} $N_K : \wedge^2 \mathsf D L \to \mathsf D L$ of $K$, vanishes. Here
\[
N_K(D_1, D_2) := [KD_1, KD_2] - [D_1, D_2] - K[KD_1, D_2] - K[D_1, KD_2],
\]
for all $D_1, D_2 \in \Gamma(\mathsf D L)$. A manifold equipped with a complex Atiyah bundle is an \textbf{odd complex manifold}.
\end{definition}

Adopting Dictionary 2 we get essentially the same answer. More precisely, if we apply
Dictionary 2 to complex manifolds we get \emph{homogeneous complex manifolds}.

\begin{definition} \label{def:2.11}
A \textbf{homogeneous complex manifold} is a homogeneous manifold $\widetilde{L}$
equipped with a \textbf{homogeneous complex structure}, i.e. an ordinary complex structure $\widetilde{K} : T\widetilde{L} \to T\widetilde{L}$ which is additionally homogeneous of degree 0:
\[
h_t^\ast \widetilde{K} = \widetilde{K}, \quad \text{for all } t \in \mathbb{R}^\ast.
\]
\end{definition}

The terminology ``odd complex manifold'' is motivated by the fact that odd complex
manifolds are necessarily odd dimensional. They can be seen as the odd dimensional
analogues of complex manifolds (whose real dimension is necessarily even). We have
the following

\begin{theorem} \label{thm:2.12}
The category of odd complex manifolds with odd holomorphic maps is
equivalent to the category of homogeneous complex manifolds with $\mathbb R^\ast$-equivariant holomorphic maps.
\end{theorem}

\begin{proof}[Proof (a sketch)]
Let $L \to M$ be a line bundle. For every complex Atiyah structure $K$
on $L$ there exists a unique complex structure $\widetilde{K}$ on $\widetilde{L}$ such that
\[
\widetilde{K D} = \widetilde{K} \widetilde D, \quad \text{for all } D \in \Gamma(\mathsf D L).
\]
The complex structure $\widetilde{K}$ is homogeneous of degree $0$, and the
assignment $K \mapsto \widetilde{K}$ establishes a bijection between complex Atiyah structures on $L$ and
homogeneous complex structures of degree $0$ on $\widetilde L$.
\end{proof}

The homogeneous complex manifold $(\widetilde{L}, \widetilde{K})$ corresponding to an odd complex manifold
$(M, L, K)$ will be called the \emph{complex cone} of $(M, L, K)$.

\begin{example} \label{ex:2.13}
Consider the cylinder $M = \mathbb{R} \times \mathbb{C}^n$. The trivial real line bundle
$\mathbb R_M$ is canonically equipped with a complex Atiyah structure. Namely, denote by
$(z^i = x^i + \sqrt{-1}y^i)$ the standard complex coordinates on $\mathbb{C}^n$ and by $u$ the standard
coordinate on the $\mathbb{R}$ factor of $M$. The module $\mathsf D \mathbb R_M$ is spanned by the identity derivation
$\mathbb{I}$ and the ordinary partial derivatives
\[
\frac{\partial}{\partial u}, \ldots, \frac{\partial}{\partial x^i}, \ldots, \frac{\partial}{\partial y^i}, \ldots
\]
and there is a unique complex Atiyah structure $K_{can} : \mathsf D\mathbb R_M \to\mathsf  D \mathbb R_M$ on $\mathbb R_M$ such that
\[
K_{can} \mathbb{I} = \frac{\partial}{\partial u} \quad \text{and} \quad K_{can} \frac{\partial}{\partial x^i} = \frac{\partial}{\partial y^i}, \quad i = 1, \ldots, n.
\]
The complex cone of $(M, \mathbb R_M, K_{can})$ is the open subset $\widetilde{M} : w \neq 0$ in the complex
manifold $\mathbb{C} \times \mathbb{C}^n$ with complex coordinates $(w = u + \sqrt{-1}v, z^i)$. The homogeneity
structure on $\widetilde{M}$ is defined by $h_t(u, v, z^i) := (u, tv, z^i)$.
\end{example}

\begin{theorem}[Odd Newlander-Nirenberg {\cite[Appendix A]{Sch2020}}] \label{thm:2.14}
Every odd complex manifold is locally of
the form $(M = \mathbb{R} \times \mathbb{C}^n, \mathbb{R}_M, K_{can})$.
\end{theorem}

Odd complex manifolds have already appeared in the Metric Contact Geometry
literature under a different name: \emph{normal almost contact manifolds} \cite{Bla2010} (beware that the
terminology \emph{almost contact structure} might refer to very different objects in different
papers). Recall that an \textbf{almost contact manifold} is a manifold $M$ equipped with an
\textbf{almost contact structure}, i.e. a triple $(\Phi, \xi, \eta)$ consisting of a $(1, 1)$-tensor $\Phi : TM \to TM$,
a vector field $\xi \in \mathfrak{X}(M)$, and a 1-form $\eta \in \Omega^1(M)$ such that
\[
\Phi^2 = -1 + \eta \otimes \xi, \quad \Phi(\xi) = 0, \quad \Phi^\ast(\eta) = 0, \quad \text{and} \quad \eta(\xi) = 1.
\]
Almost contact manifolds have been introduced in the literature as odd dimensional
analogues of almost complex manifolds. In fact, an almost contact structure $(\Phi, \xi, \eta)$
defines a fiber-wise complex structure $K_{(\Phi, \xi, \eta)}$ on the gauge algebroid $\mathsf D \mathbb R_M$ of the trivial
line bundle as follows. A derivation of $\mathbb R_M$ can be seen as a pair $(X, r)$ where $X$ is a
vector field, and $r$ is a smooth function on $M$. We put
\[
K_{(\Phi, \xi, \eta)} (X, r) = \big(\Phi(X) - r\xi, \eta(X)\big).
\]
The fiber-wise complex structure $K_{(\Phi, \xi, \eta)}$ is a complex Atiyah structure (i.e. its Lie algebroid
Nijenhuis torsion vanishes) if and only if $(\Phi, \xi, \eta)$ is a \emph{normal almost contact structure}, i.e.
\[
N_\Phi + d\eta \otimes \xi = 0, \quad d\eta(\Phi-, -) + d\eta(-, \Phi-) = 0, \quad \mathcal L_\xi \Phi = 0, \quad \text{and} \quad \mathcal L_\xi \eta = 0.
\]
Moreover, every odd complex manifold $(M, L, K)$ is locally of this form, up to line bundle isomorphisms. We believe that odd complex manifolds are conceptually more appropriate (and much simpler) than normal
almost contact manifolds as odd dimensional analogues of complex manifolds.

\section{\label{sec:hgs}Homogeneous $G$-Structures}

The example in the previous section shows that
various kinds of geometric structures can be translated via the Symplectic-to-Contact Dictionary (either version
1 or 2). It is then natural to wonder whether or not the dictionary can be applied
to a generic $G$-structure. This question has been answered in \cite{Tor2020} which is our reference for this section. It turns out that, to work with $G$-structures, the second version of the
Dictionary is more appropriate. The final outcome is, roughly, that, for any $G$, there
is a family of homogeneity conditions that we can impose on a $G$-structure, each one
providing a different dictionary \cite{Tor2020}. Our plan, in this section, is as follows.
We begin analyzing the case $G = \mathrm{Sp}(q, \mathbb R)$, the symplectic group. In this case, a $G$-structure
is an almost symplectic structure. But a homogeneous symplectic manifold is essentially the same
as a contact manifold. We express the homogeneity property purely in the language of
$G$-structures, and then axiomatize to get a general definition valid for every $G$.

Let $G \subseteq \mathrm{GL}(n, \mathbb R)$ be a Lie subgroup of the general linear group and let $N$ be an
$n$-dimensional manifold. Recall that a \textbf{$G$-structure} on $N$ is a \textbf{$G$-reduction} of the frame
bundle $\mathrm{Fr}(N) \to N$, i.e. a $G$-invariant subbundle $P \subseteq \mathrm{Fr}(N)$ such that $P \to N$, with
the restricted $G$-action, is a principal $G$-bundle. All our actions will be from the right, a
point in $\mathrm{Fr}(N)$ over a point $x \in N$ is a linear isomorphism $\mathbb{R}^n \to T_x N$ and $\mathrm{GL}(n, \mathbb{R})$ acts on $\mathrm{Fr}(N)$ by pre-composition: for all $g \in GL(n, \mathbb{R})$ and all frames $\psi : \mathbb{R}^n \to T_x N$, the
action $\psi \cdot g$ is the frame $\psi \circ g$, where we interpret $g$ as a linear isomorphism $g : \mathbb{R}^n \to \mathbb{R}^n$.
Now, let $n = 2k$, and let $G = \mathrm{Sp} (k, \mathbb R) \subseteq GL(n, \mathbb{R})$ be the symplectic group. In this case a $G$-structure $P \subseteq \mathrm{Fr}(N)$ is always the space of symplectic frames of an \emph{almost symplectic
structure} on $N$ (i.e.~a non-degenerate $2$-form). Conversely, the space of symplectic frames of any almost symplectic structure is always an
$\mathrm{Sp} (k, \mathbb R)$-structure. Things go through in a similar way for almost complex structures (in which case $G$ is the complex group), Riemannian structures ($G$ is the orthogonal group), etc.

Let's go back to the case $G = \mathrm{Sp} (k, \mathbb R)$. Assume that $N = \widetilde L$ is a homogeneous manifold
over a manifold $M$ (of dimension $2k-1$), and let $\widetilde{\omega}$ be a homogeneous symplectic
structure on $N$. We know that $\widetilde{\omega}$ corresponds to a contact structure $H$ on $M$. Fix a
point $x \in M$ and choose Darboux coordinates $(x^i, u, p_i)$ for $H$. The section
\[
\lambda := \frac{\partial}{\partial u} \mathbin{\mathrm{mod}} H \in \Gamma(L)
\]
of the normal bundle $L = TM/H$ is a local generator. It is now easy to see that
\[
\left(x^i, - \widetilde \lambda p_i, u, \widetilde \lambda\right)
\]
are Darboux coordinates for $\widetilde{\omega}$, hence
\begin{equation}\label{eq:Psi}
\Psi = \left( \frac{\partial}{\partial x^i}, - \frac{\partial}{\partial \widetilde \lambda p_i}, \frac{\partial}{\partial u}, \frac{\partial}{\partial \widetilde \lambda} \right)
\end{equation} 
is a section of the $\mathrm{Sp} (k, \mathbb R)$-structure $P$ corresponding to $\widetilde{\omega}$. This section satisfies the
following homogeneity property: the diagram
\begin{center}
\begin{tikzcd}
\mathbb{R}^{2k+2} \arrow{r}{\Psi_\varepsilon} \arrow[d, "A(t)"'] & T_\varepsilon \widetilde{L} \arrow{d}{d_\varepsilon h_t} \\
\mathbb{R}^{2k+2} \arrow{r}{\Psi_{t\varepsilon}} & T_{t\varepsilon} \widetilde{L}
\end{tikzcd}
\end{center}
commutes for all $\varepsilon \in \widetilde{L}$ and all $t \in \mathbb{R}^*$, where
\begin{equation}\label{eq:A}
A(t) = \begin{pmatrix} t\mathbb{I} & 0 \\ 0 & \mathbb{I} \end{pmatrix}.
\end{equation}
It is straightforward to check that the map $A : \mathbb{R}^\ast \to \mathrm{GL} (n, \mathbb R)$ takes values in the normalizer
of $\mathrm{Sp} (k, \mathbb R)$ in $\mathrm{GL} (n, \mathbb R)$, and the $\mathrm{Sp} (k, \mathbb R)$-structure $P$ is preserved by the $\mathbb{R}^\ast$-action $h^A$ on $\mathrm{Fr}(\widetilde{L})$
given by
\[
h^A : \mathbb{R}^\ast \times \mathrm{Fr}(\widetilde{L}) \to \mathrm{Fr}(\widetilde{L}), \quad h^A_t(\psi) := d h_t \circ \psi \circ A(t)^{-1}.
\]
Our strategy is now to adopt the last observations as axioms. So, let $\widetilde{L}$ be a homogeneous
manifold (of dimension $n$), and let $\Psi$ be a section of the frame bundle $\mathrm{Fr}(\widetilde{L})$. It is clear
that, for every $t \in \mathbb{R}^\ast$ and $\varepsilon \in \widetilde{L}$, there exists a unique invertible matrix $A_\Psi (t, \varepsilon) \in \mathrm{GL} (n, \mathbb R)$
such that
\[
d_\varepsilon h_t \circ \Psi_\varepsilon = \Psi_{t\varepsilon} \circ A_\Psi(t, \varepsilon).
\]
\begin{propdef}\label{def:2.15}
A section $\Psi$ of the frame bundle $\mathrm{Fr}(\widetilde{L})$ is \textbf{homogeneous} if the matrix valued map $A_\Psi = A_\Psi(t, \varepsilon)$ is independent of $\varepsilon$. In this case $A_\Psi : \mathbb{R}^\ast \to \mathrm{GL} (n, \mathbb R)$
is a Lie group homomorphism, called the \textbf{degree} of $\Psi$.
\end{propdef}

\begin{example} \label{ex:2.16}
The coordinate frame \eqref{eq:Psi} is homogeneous of degree $A$ given by \eqref{eq:A}.
\end{example}

\begin{propdef}\label{def:2.17}
Let $P \subseteq \mathrm{Fr}(\widetilde{L})$ be a $G$-structure on a homogeneous
manifold $\widetilde{L}$ over a connected manifold $M$. The following conditions are equivalent:
\begin{enumerate}
    \item locally, around every point, there exists a homogeneous section $\Psi_0 \in \Gamma(P)$ such
    that $A_{\Psi_0} : \mathbb{R}^\ast \to  \mathrm{GL} (n, \mathbb R)$ takes values in the normalizer $N(G) \subseteq  \mathrm{GL} (n, \mathbb R)$ of $G$ in $ \mathrm{GL} (n, \mathbb R)$;
    \item locally, around every point, there exists a homogeneous section $\Psi_0 \in \Gamma(P)$ such
    that $P$ is preserved under the induced action $h^{A_{\Psi_0}}$ of $\mathbb{R}^\ast$ on $\mathrm{Fr}(\widetilde{L})$;
    \item locally, around every point, there exists a homogeneous section $\Psi \in \Gamma(P)$ and,
    for all such $\Psi$, the homomorphism $A_\Psi : \mathbb{R}^\ast \to  \mathrm{GL} (n, \mathbb R)$ takes values in $N(G)$;
    \item locally, around every point, there exists a homogeneous section $\Psi \in \Gamma(P)$ and,
    for all such $\Psi$, the $G$-structure $P$ is preserved under $h^{A_\Psi}$.
\end{enumerate}
If one, hence all, of the above conditions is satisfied, we say that $P$ is a \textbf{homogeneous
$G$-structure}. In this case the composition
\[
\alpha: \mathbb{R}^\ast \xrightarrow{A_\Psi} N(G) \to N(G)/G
\]
is independent of the local homogeneous section $\Psi \in \Gamma(P)$ and it is called the \textbf{degree}
of $P$. We also say that $P$ is \textbf{$\alpha$-homogeneous}.
\end{propdef}

So, given a Lie subgroup $G \subseteq  \mathrm{GL} (n, \mathbb R)$, in order to discuss homogeneous $G$-structures,
we need to classify the possible degrees first, computing the normalizer $N(G)$ and the
homomorphisms $\alpha : \mathbb{R}^\ast \to N(G)/G$. We conclude this section illustrating this with
three examples.

\begin{example} \label{ex:2.18}
Let $G =  \mathrm{Sp} (k, \mathbb R)$ be the symplectic group. In this case, there is a canonical
isomorphism $N(G)/G = N( \mathrm{Sp} (k, \mathbb R))/\mathrm{Sp} (k, \mathbb R) \cong \mathbb{R}^\ast$, and there are two main homomorphisms
$\alpha : \mathbb{R}^\ast \to N(\mathrm{Sp} (k, \mathbb R))/\mathrm{Sp} (k, \mathbb R) \cong \mathbb{R}^\ast$, namely the identity homomorphism $\operatorname{id}$ and the trivial
(constant) homomorphism $1$. One can show that homogeneous manifolds with a homogeneous $\mathrm{Sp} (k, \mathbb R)$-structure $P$ of degree $\alpha = \operatorname{id}$ are equivalent to manifolds $M$ with a pair
$(H, \Upsilon)$ where $H$ is a hyperplane distribution on $M$, and $\Upsilon$ is an $L = TM/H$-valued
$2$-form on $M$ such that $\kappa_H - \Upsilon$ is non-degenerate \cite[Section 5]{Tor2020}. Additionally $P$ is an integrable
$G$-structure if and only if $\Upsilon = 0$, hence $H$ is a contact structure.
When $\alpha = 1$ we recover a class of geometric structures called (\emph{almost}) \emph{cosymplectic
structures} \cite[Section 6.1]{Tor2020} (see, e.g., \cite{Cap2013} for a survey of Cosymplectic Geometry). Notice that Cosymplectic Geometry is an odd dimensional analogue of Symplectic Geometry as well and, actually, there exists a Symplectic-to-Cosymplectic Dictionary very similar to the Symplectic-to-Contact dictionary discussed in this note.
\end{example}

\begin{example} \label{ex:2.19}
Let $G = \mathrm{GL} (k, \mathbb C) \subseteq \mathrm{GL} (2k, \mathbb R)$ be the complex general linear group
(seen as a Lie subgroup of the real general linear group). In this case, there is a
canonical isomorphism $N(G)/G = N(\mathrm{GL} (k, \mathbb C))/\mathrm{GL} (k, \mathbb C) \cong \mathbb{Z}_2$. Consider the constant homomorphism $0 : \mathbb{R}^\ast \to \mathbb{Z}_2$. One can show that homogeneous manifolds with a homogeneous $\mathrm{GL} (k, \mathbb C)$-structure $P$ of degree $\alpha = 0$ are equivalent to almost
complex Atiyah bundles $(M, L, K)$ \cite[Section 6.2]{Tor2020}. Additionally $P$ is an integrable $G$-structure if and only if $K$ is a complex structure.
\end{example}

\begin{example} \label{ex:2.20}
Let $G = \mathrm{O} (n, \mathbb R)$ be the orthogonal group. In this case, there is a canonical
isomorphism $N(G)/G = N(\mathrm{O} (n, \mathbb R))/\mathrm{O} (n, \mathbb R) \cong \mathbb{R}^\ast_+$, the multiplicative group of positive reals. Consider the homomorphism $\alpha:\mathbb{R}^\ast \to \mathbb{R}^\ast, r \mapsto |r|^{1/2}$. One can show that homogeneous manifolds with a homogeneous $\mathrm{O} (n, \mathbb R)$-structure $P$ of degree $\alpha$ are equivalent to manifolds $M$ equipped with a Riemannian metric $g$ and a 1-form $\eta \in \Omega^1(M)$ \cite[Section 6.3]{Tor2020}. Additionally $P$ is an integrable $G$-structure if and only if $g, \eta$ satisfy an appropriate system of PDEs. When $\eta = 0$ this system boils down to $g$ having
constant sectional curvature equal to $1/4$.
\end{example}

\end{document}